\documentclass[12pt]{article}

\usepackage{amsmath}
\usepackage{amssymb}

\newcommand{\qed}{$\;\;\;\Box$}
\newenvironment{proof}{\par\smallbreak{\sl\bf Proof.~}}
{\unskip\nobreak\hfill \qed \par\medbreak}

\newcounter{claim}
\renewcommand{\theclaim}{\arabic{claim}}
{\par\medskip\par}

{\qed\par\smallbreak}
\newcommand{\D}{{\mathcal D}}

\newcommand{\N}{{\mathbb N}}
\newcommand{\R}{{\mathbb R}}
\newcommand{\C}{{\mathbb C}}
  \newcommand{\Z}{{\mathbb Z}}
\newcommand{\M}{{\mathbb M}}

\newcommand{\A}{{\mathcal A}}
\newcommand{\B}{{\mathcal B}}
\newcommand{\LL}{{\mathcal L}}

\newcommand{\beq}{\begin{equation}}
\newcommand{\ee}{\end{equation}}

\renewcommand{\d}{\partial}

\newtheorem{thm}{Theorem}[section]
\newtheorem{lem}[thm]{Lemma}

\newtheorem{defn}[thm]{Definition}

\newcommand{\al}{\alpha}
\newcommand{\be}{\beta}
\newcommand{\ga}{\gamma}

\newcommand{\vphi}{\varphi}

\newcommand{\om}{\omega}

\newcommand{\reff}[1]{(\ref{#1})}

\newcommand{\diag}{\mathop{\rm diag}\nolimits}
\newcommand{\ess}{\mathop{\rm ess}}
\title{
Smoothing effect and Fredholm property for first-order hyperbolic PDEs
} 

\newcounter{thesame}
\author{
I.~Kmit\\
{\small
Institute of Mathematics, Humboldt University of Berlin,}
\\
{\small Rudower Chaussee 25, D-12489 Berlin, Germany }
\\
{\small
and Institute for Applied Problems of Mechanics and Mathematics, }
\\
{\small
Ukrainian Academy of Sciences,  Naukova St.\ 3b, 79060 Lviv,
Ukraine 
}
\\
{\small   E-mail:
{\tt kmit@informatik.hu-berlin.de}}
}
 
\date{}

\begin{document}
\maketitle

\begin{abstract}
We give an exposition of recent results on regularity and Fredholm properties for first-order 
one-dimensional hyperbolic PDEs. We show that large classes of boundary operators cause an  
effect that smoothness increases with time. This property is the key in finding regularizers 
(parametrices) for hyperbolic problems. We construct regularizers for periodic problems 
for dissipative first-order linear hyperbolic PDEs and  show that these problems are  modeled by  
Fredholm operators of index zero. 
\end{abstract}




\section{Introduction}\label{sec:intr} 
\renewcommand{\theequation}{{\thesection}.\arabic{equation}}
\setcounter{equation}{0}

In contrast to ODEs and parabolic PDEs, the Fredholm property and regularity behavior
of hyperbolic problems are much less understood. In a recent series of papers~\cite{Km,KmRe1,KR2},
the latter two written jointly with Lutz Recke, we undertook a detailed analysis of this subject
for first-order one-dimensional hyperbolic operators. The purpose of the present survey paper 
is to present  some of our results and their extensions with emphasize on the
smoothing phenomenon, construction of parametrices, and the Fredholmness of index zero.

An important step in local investigations of nonlinear differential equations 
(many ODEs and parabolic PDEs) is to establish 
the Fredholm solvability of their linearized versions. In the hyperbolic case 
this step is much more involved.
Since the singularities of (semi-)linear hyperbolic equations propagate along characteristic curves,
a solution cannot be more regular in the entire time-space domain than it is on the boundary.
It can even be less regular which is known as the \emph{loss-of-smoothness} effect.
 Therefore the Fredholm analysis of hyperbolic problems requires establishing an optimal
regularity relation between  the spaces of solutions and right-hand sides of the differential 
equations.
 
Proving a Fredholm solvability is typically based on the basic fact that any Fredholm operator is exactly a compact 
perturbation of a bijective operator. In the hyperbolic case, using the compactness argument gets complicated 
because of the lack of regularity over the \emph{whole} time-space domain.

Our approach is based on the fact that for a range of boundary operators, solutions improve smoothness
\emph{dynamically}, more precisely, they eventually become $k$-times continuously differentiable for each particular $k$.
We prove such kind of results in Section~\ref{sec:smooth}. Note that in some interesting cases
the smoothing phenomenon was shown earlier in~\cite{Elt,Hillen,LavLyu,Lyulko}.

This phenomenon allows us in Section~\ref{sec:fredh} to work out 
a regularization procedure
via construction of a parametrix. We here present a quite general approach 
to proving the Fredholmness for first-order dissipative
hyperbolic PDEs and apply it to the periodic problems. Our Fredholm results  cover 
non-strictly hyperbolic systems with discontinuous coefficients, but they are new even in the case of
strict hyperbolicity and smooth coefficients.

From a more general perspective, the smoothing effect and Fredholmness properties 
 play an important role in the study of the Hopf bifurcation  and 
periodic synchronizations in nonlinear hyperbolic PDEs~\cite{akr,KR1}  via the 
Implicit  Function Theorem and Lyapunov-Schmidt procedure~\cite{ChowHale,Ki} and averaging procedure~\cite{Bogo,Samo}. 

From the practical point of view, our techniques  cover the so-called 
traveling-wave models from laser dynamics~\cite{LiRadRe,RadWu} 
(describing the appearance of
self-pulsations of lasers and modulation of stationary laser states by time periodic electric pumping), 
population dynamics~\cite{cush,hiha,webb},  and
chemical kinetics~\cite{AkrBelZel,aris1,aris2,Zel1} (describing mass transition in terms
of convective diffusion and chemical reaction and analysis of chemical processes in counterflow 
chemical reactors).

\section{Smoothing effect}\label{sec:smooth}

Here we describe  classes of 
(initial-)boundary problems for first-order one-dimensional hyperbolic PDEs 
whose solutions improve their regularity in time. 

Set
$$
\Pi_T = \{(x,t)\,:\,0<x<1, T<t<\infty\}.
$$ 
We address the problem
\begin{equation}\label{eq:1}
(\partial_t  + a(x,t)\partial_x + b(x,t))
 u = f(x,t), 
\end{equation}
\begin{eqnarray}
&&u(x,0) = \varphi(x),
\label{eq:2}\\&&
\begin{array}{ll}
u_j(0,t) = (Ru)_j(t),\quad 1\le j\le m 
\\
u_j(1,t) = (Ru)_j(t), \quad m<j\le n
\end{array}\label{eq:3}
\end{eqnarray}
in the semi-strip $\Pi_0$ and the problem \reff{eq:1}, \reff{eq:3} in the strip $\Pi_{-\infty}$.
Here  $u=(u_1,\ldots,u_n)$, $f=(f_1,\ldots,f_n)$, and $\varphi=(\vphi_1,\dots,\vphi_n)$ are  vectors
of real-valued functions, 
$b=\{b_{jk}\}_{j,k=1}^n$ and $a=\diag(a_1,\dots,a_n)$ are matrices of real-valued functions, and 
$0\le m\le n$ are fixed integers.
Furthermore, $R$ is an operator mapping $C\left(\overline\Pi_0\right)^n$ into $C\left([0,\infty)\right)^n$, 
 and similarly for $R$ in $\Pi_{-\infty}$.
In Sections~\ref{sec:classical}--\ref{sec:dis} we give examples of $R$
as representatives of some classes of boundary operators ensuring  smoothing  
solutions.

In the domain under consideration we assume that 
\begin{equation}\label{eq:L1}
a_j>0 \mbox{ for all } j\le m\quad \mbox{ and }\quad a_j<0 \mbox{ for all } j>m,
\end{equation}
\begin{equation}\label{eq:L2}
\inf_{x,t}|a_j|>0 \mbox { for all } j\le n,
\end{equation}
and 
\beq
\label{cass}
\begin{array}{l}
\mbox{for all } 1 \le j \not= k \le n \mbox{ there exists }  p_{jk} \in C^1([0,1]\times \R)
\\\mbox{such that } 
b_{jk}=p_{jk}(a_k-a_j)  \mbox{ and } p_{jk}=0 \\ \mbox{in the interior of the domain }
\{(x,t)\,:\,a_j(x,t)=a_k(x,t)\}.
\end{array}
\ee
Note that all these conditions are not restrictive neither from the practical nor from the theoretical  points of view. In particular, 
condition \reff{eq:L1} is true in traveling-wave models of laser and population dynamics as well as chemical kinetics,
where the functions $u_j$ for $j\le m$ (respectively, $m+1\le j\le n$) describe  ``species'' traveling to
the right (respectively, to the left). Condition \reff{eq:L2} means that all characteristics of the system (\ref{eq:1})
are bounded and the system (\ref{eq:1}) is, hence, non-degenerate. Finally, the condition  \reff{cass} is a kind of Levy condition 
usually appearing to compensate non-strict hyperbolicity where the coefficients $a_j$ and $a_k$ for some $j\ne k$
coincide at least at one point, say, $(x_0,t_0)$. In this case the 
lower-order terms with the coefficients $b_{jk}$ and $b_{kj}$ contribute  to the  system 
at  $(x_0,t_0)$ longitudinally 
to characteristic directions (keeping responsibility  for the propagation of singularities), while in
the strictly hyperbolic case we have a qualitatively different transverse contribution
at that point. The purpose of \reff{cass} is to suppress propagation of singularities 
through the non-diagonal lower-order terms of (\ref{eq:1}).

We will impose the following smoothness assumptions on the initial data: The entries of $a$,
$b$, and $f$ are $C^\infty$-smooth in all their arguments in the respective domains,
while the entries of $\vphi$ are assumed to be continuous functions only. 

Let us introduce the system resulting
from (\ref{eq:1})--(\ref{eq:3}) (resp., from  (\ref{eq:1}), (\ref{eq:3})) 
via integration along characteristic curves. 
 For given $j\le n$, $x \in [0,1]$, and $t \in \R$, the $j$-th characteristic of \reff{eq:1} 
passing through the point $(x,t)$ is defined 
as the solution $\xi\in [0,1] \mapsto \om_j(\xi;x,t)\in \R$ of the initial value problem

\beq\label{char}
\partial_\xi\om_j(\xi;x,t)=\frac{1}{a_j(\xi,\om_j(\xi;x,t))},\;\;
\om_j(x;x,t)=t.
\ee
Define
\begin{eqnarray*}
c_j(\xi,x,t)=\exp \int_x^\xi
\left(\frac{b_{jj}}{a_{j}}\right)(\eta,\om_j(\eta;x,t))\,d\eta,\quad
d_j(\xi,x,t)=\frac{c_j(\xi,x,t)}{a_j(\xi,\om_j(\xi;x,t))}.
\end{eqnarray*}
Due to \reff{eq:L2}, the characteristic curve $\tau=\om_j(\xi;x,t)$ reaches the
boundary of $\Pi_T$ in two points with distinct ordinates. Let $x_j(x,t)$
denote the abscissa of that point whose ordinate is smaller.
Straightforward calculations show that a $C^1$-map $u: [0,1]\times[0,\infty) \to \R^n$ is a solution to 
(\ref{eq:1})--(\ref{eq:3}) if and only if
it satisfies the following system of integral equations
\begin{eqnarray}
\label{rep}
\lefteqn{
u_j(x,t)=(BSu)_j(x,t)}\nonumber\\
&&-\int_{x_j(x,t)}^x d_j(\xi,x,t)\sum_{k=1\atop k\not=j}^n b_{jk}(\xi,\om_j(\xi;x,t))u_k(\xi,\om_j(\xi;x,t))d\xi\nonumber\\ 
&&+\int_{x_j(x,t)}^x d_j(\xi,x,t)f_j(\xi,\om_j(\xi;x,t))d\xi,\quad j\le n,
\end{eqnarray}
where
\begin{eqnarray}
\label{B}
(Bu)_j(x,t)=c_j(x_j(x,t),x,t)u_j\left(x_j(x,t),\om_j(x_j(x,t);x,t)\right),
\end{eqnarray}
\begin{eqnarray}
\label{S}
(Su)_j(x,t)=
\begin{cases}
(Ru)_j(t) & \text{if $ t>0, $} \\
\vphi_j(x)      & \text{if $t=0.$}
\end{cases}
\end{eqnarray}
Here $B$ is a shifting operator from $\d\Pi_0$ along
characteristic curves of (\ref{eq:1}), while the operator 
$S$ is used to denote the boundary operator on the whole $\d\Pi_0$. Similarly,
a $C^1$-map $u: [0,1]\times\R \to \R^n$ is a solution to 
(\ref{eq:1}), (\ref{eq:3}) if and only if
it satisfies the system \reff{rep}, where the definition of $S$ is changed to $S=R$.

This motivates the following definition:
\begin{defn}\label{defn:cont}\rm
{\bf (1)} A continuous  function $u$ is called a continuous solution to  (\ref{eq:1})--(\ref{eq:3}) in $\overline\Pi_0$
if it satisfies \reff{rep} with $S$ defined by \reff{S}.

{\bf (2)} A continuous  function $u$ is called a continuous solution to  (\ref{eq:1}), (\ref{eq:3}) in $\overline\Pi_{-\infty}$
if it satisfies \reff{rep} with $S=R$.
\end{defn}

Existence results for (continuous) solutions to the problems under consideration are obtained in~\cite{AbMy,serb,ijdsde,KmHo}.

\begin{defn}\label{defn:smoothing}\rm
 A solution $u$ to the problem (\ref{eq:1})--(\ref{eq:3}) or (\ref{eq:1}), (\ref{eq:3}) is called {\it smoothing} 
if, for every $k\in\N$, there exists $T>0$ such that $u_j\in C^k\left(\overline\Pi_T\right)$ for all
$j\le n$.
\end{defn}

For the initial-boundary value problem (\ref{eq:1})--(\ref{eq:3}) Definition~\ref{defn:smoothing} reflects 
a dynamic nature of the smoothing property stating 
that the regularity of solutions increases in time. The fact that the regularity cannot be uniform in the entire
domain is a straightforward consequence of the propagation of singularities along characteristic curves.
Moreover, switching from $C^k$ to $C^{k+1}$-regularity is jump-like; this phenomenon is usually observed 
in the situations when solutions of hyperbolic PDEs change their 
regularity (see e.g.,~\cite{Lyulko,Ober86,popivanov03,RauchReed81}).

Note that, if the problem (\ref{eq:1}), (\ref{eq:3}) is subjected to periodic conditions in~$t$, then
Definition~\ref{defn:smoothing} implies that the smoothing solutions immediately meet the $C^\infty$-regularity in the
entire domain.

Definition~\ref{defn:smoothing} captures the general nature of the smoothing phenomenon for hyperbolic PDEs.
A more precise information can be extracted from the proof of 
Theorems~\ref{thm:classical}, \ref{thm:population}, and \ref{thm:dissipative} below: 
Reaching  the $C^k$-regularity for solutions
needs only a $C^{k+1}$-regularity for $a$, $b$, and $f$. More exact regularity conditions for
the boundary data, which also depend on $k$, can be derived from these proofs as well.
These refinements are useful in some applications. 

Definition~\ref{defn:smoothing} can be strengthened by admitting
 worse regularities for the initial data. One extension of this kind, when the initial data are 
strongly singular distributions concentrated at a finite number of points, can be found in~\cite{Km}.
In~\cite{Km} we used a delta-wave solution concept. Another result in this direction~\cite{KmRe1,KR2}
concerns periodic problems and uses a variational  setting of the problem
(see also Theorem~\ref{thm:fredh} (ii)). In~\cite{KmRe1,KR2} we get an improvement of the solution regularity from 
being functionals to being functions.

In what follows we demonstrate the smoothing effect on generic examples of  large classes of boundary operators
and  show which kinds of problems can be covered by our techniques. 
Our approach to establishing smoothing results is based on the consideration of the
integral representation of the problems and the observation that
the boundary and the integral parts of this 
representation  have different influence  on the regularity of solutions.
Our main idea is to show that the integral part has a ``self-improvement'' property, while in many interesting cases
the boundary part is not responsible for propagation of singularities. The latter  contrasts to the case 
of the Cauchy problem  where the solutions cannot be smoothing as the boundary term all the time "remembers"
the regularity of the initial data. It is worthy to note that in the case of 
the problem (\ref{eq:1})--(\ref{eq:3}) in~$\Pi_0$ the 
domain of influence of the initial conditions is determined by both
parts of the integral system and is in general infinite. This makes the smoothing effect non-obvious.

\subsection{Classical boundary conditions}\label{sec:classical}

Here we specify conditions (\ref{eq:3}) to
\begin{eqnarray}
\begin{array}{ll}
u_j(0,t) = h_j(t), & 1\le j\le m, \ \  
\\
u_j(1,t) = h_j(t), & m<j\le n. \quad\quad\ \, 
\end{array}&&\label{eq:3''}
\end{eqnarray}
and consider the problem (\ref{eq:1}), (\ref{eq:2}), \reff{eq:3''}.

\begin{thm}\label{thm:classical}
 Assume that the data $a_j$, $b_{jk}$, $f_j$, and $h_j$ are 
smooth in all their arguments and $\vphi_j$ are continuous functions.
Assume also  
 (\ref{eq:L1}),  (\ref{eq:L2}), and \reff{cass}. Then 
any continuous solution to the problem (\ref{eq:1}), (\ref{eq:2}), (\ref{eq:3''})  
is smoothing.
\end{thm}

Note that in the case of smooth classical
 boundary conditions (\ref{eq:3''}),
the domain of influence of the initial data $\vphi(x)$ on $u_i$ for every $i\le n$ in general is unbounded
 (due to the lower-order terms in (\ref{eq:1})).
In spite of this, the influence of the initial data on the regularity of $u$ becomes
weaker and weaker in time causing the smoothing effect.

\begin{proof}
Suppose that $u$  is a continuous solution to the problem (\ref{eq:1})--(\ref{eq:3})
and show that the operator of the problem improves the regularity 
of $u$ in time. The idea of the proof is similar to~\cite{Km}.

We start with  an operator representation of $u$. To this end, introduce 
 linear bounded operators $D, F: C\left(\overline\Pi_0\right)^n \to C\left(\overline\Pi_0\right)^n$
 by
\begin{eqnarray*}
\left(Du\right)_j(x,t) & = & 
-\int_{x_j(x,t)}^x d_j(\xi,x,t)\sum_{k=1\atop k\not=j}^n\left(b_{jk}u_k\right)(\xi,\om_j(\xi;x,t))d\xi,\\
\left(Ff\right)_j(x,t)&=&\int_{x_j(x,t)}^x d_j(\xi,x,t)f_j(\xi,\om_j(\xi;x,t))d\xi.
\end{eqnarray*}
Note that $Ff$ is a smooth function in $x,t$.
In this notation the integral system  (\ref{rep}) can be written as
\beq\label{abstr}
u=BSu+Du+Ff.
\ee
It follows that
\beq\label{io}
u=BSu+(DBS+D^2)u+(I+D)Ff.
\ee

In the first step we prove that the right hand side of \reff{io} restricted to $\overline\Pi_{T_1}$ for some $T_1>0$
is continuously differentiable in $t$. 
The $C^1\left(\overline\Pi_{T_1}\right)^n$-regularity of $u$ will then follow from the fact that
$u$ given by \reff{rep} satisfies (\ref{eq:1}) in the distributional sense. 
By the assumption \reff{eq:L2}, we can fix a large enough  $T_1>0$ such that the 
operator $S$ in the right-hand side of 
\reff{io} restricted to $\overline\Pi_{T_1}$ does not depend on $\vphi$ and, hence,
$Su=Ru=h$, where $h=(h_1,\dots,h_n)$. 
We therefore arrive at  the equality
\beq\label{io1}
u|_{\overline\Pi_{T_1}}=Bh+DBh+D^2u+(I+D)Ff,
\ee
where $u|_{\overline\Pi_{T_1}}$ denotes the restriction of $u$ to $\overline\Pi_{T_1}$.
By the regularity assumption on $a,b,f$, and $h$, 
 the function $Bh+DBh+(I+D)Ff$ is smooth.
We have reduced the problem to show that the operator $D^2$ is smoothing, more specifically, that $D^2u$ is $C^1$-smooth 
in $t$ on $\overline\Pi_{T_1}$. 

Notice that for $t\ge T_1$ the function $x_j(x,t)$ is a constant depending only on~$j$. Below we therefore will drop 
the dependence of $x_j$ on $x$ and $t$.
Fix a sequence $u^l\in C^1\left(\overline\Pi_0\right)^n$ such that 
\beq\label{eq:lim_0}
u^l\to u \mbox{ in } C\left(\overline\Pi_0\right)^n \mbox{ as } l\to\infty. 
\ee
By convergence in $C\left(\Omega\right)^n$ here and below we mean the
 uniform convergence on any compact subset of $\Omega$.
Then $D^2u^l\to D^2u$ in $C\left(\overline\Pi_0\right)^n$ as well. It suffices to prove 
that $\d_t\left[D^2u^l\right]$ converges in  $C\left(\overline\Pi_{T_1}\right)^n$ as $l\to\infty$. 
Given $j\le n$, consider the following expression for $\left(D^2u^l\right)_j(x,t)$,
obtained by change of the order of integration:
\begin{eqnarray}
\lefteqn{
\left(D^2u^l\right)_j(x,t)}\label{D11}\\
&&=\sum_{k=1\atop k\not=j}^n\sum_{i=1\atop i\not=k}^n
\int_{x_j}^x \int_\eta^x d_{jki}(\xi,\eta,x,t)b_{jk}(\xi,\om_j(\xi;x,t))u_i^l(\eta,\om_k(\eta;\xi,\om_j(\xi;x,t))) d \xi d \eta
\nonumber
\end{eqnarray}
with
\begin{eqnarray*}
d_{jki}(\xi,\eta,x,t)
=d_j(\xi,x,t)d_k(\eta,\xi,\om_j(\xi;x,t))b_{ki}(\eta,\om_k(\eta;\xi,\om_j(\xi;x,t))).
\label{djkl}
\end{eqnarray*}
It follows that
\begin{eqnarray}
\lefteqn{
\d_t\left[\left(D^2u^l\right)_j(x,t)\right]
}\nonumber\\\nonumber
&=&\sum_{k=1\atop k\not=j}^n\sum_{i=1\atop i\not=k}^n
\int_{x_j}^x \int_\eta^x \d_t\bigl[d_{jki}(\xi,\eta,x,t)b_{jk}(\xi,\om_j(\xi;x,t))\bigr]
u_i^l(\eta,\om_k(\eta;\xi,\om_j(\xi;x,t))) d \xi d \eta\nonumber\\
&+&\sum_{k=1\atop k\not=j}^n\sum_{i=1\atop i\not=k}^n
\int_{x_j}^x \int_\eta^x {d}_{jki}(\xi,\eta,x,t)b_{jk}(\xi,\om_j(\xi;x,t))\nonumber\\
&\times&\d_3\om_k(\eta;\xi,\om_j(\xi;x,t))\d_t\om_j(\xi;x,t)\d_2u_i^l(\eta,\om_k(\eta;\xi,\om_j(\xi;x,t))) d \xi d \eta,\label{dtD}
\end{eqnarray}
where $\d_rg$ here and below  denotes the derivative of $g$ with respect to the $r$-th argument. 
The first summand in the right-hand side converges in $C\left(\overline\Pi_{T_1}\right)$.
Our task is therefore reduced to show the uniform convergence 
of all integrals in the second summand, whenever $(x,t)$ varies on a compact subset of $\overline\Pi_{T_1}$. For this purpose we will transform the integrals as follows.
First note that the assumption \reff{cass} determines the function $p_{jk}$ uniquely.
Using  \reff{cass} and the formulas
\begin{eqnarray}
\d_x\om_j(\xi;x,t) & = & -\frac{1}{a_j(x,t)} \exp \int_\xi^x \left(\frac{\d_ta_j}{a_j^2}\right)(\eta,\om_j(\eta;x,t)) d \eta,\\
\label{dt}
\d_t\om_j(\xi;x,t) & = & \exp \int_\xi^x \left(\frac{\d_ta_j}{a_j^2}\right)(\eta,\om_j(\eta;x,t)) d \eta,
\end{eqnarray}
we get
\begin{eqnarray*}
\lefteqn{\int_{x_j}^x \int_\eta^x {d}_{jki}(\xi,\eta,x,t)b_{jk}(\xi,\om_j(\xi;x,t))}\\
&\times&\d_3\om_k(\eta;\xi,\om_j(\xi;x,t))\d_t\om_j(\xi;x,t)\d_2u_i^l(\eta,\om_k(\eta;\xi,\om_j(\xi;x,t))) d \xi d \eta\\
&=&\int_{x_j}^x \int_\eta^x {d}_{jki}(\xi,\eta,x,t)\d_3\om_k(\eta;\xi,\om_j(\xi;x,t))\d_t\om_j(\xi;x,t)\\
&\times&b_{jk}(\xi,\om_j(\xi;x,t))\Bigl[\bigl(\d_\xi\om_k\bigr)(\eta;\xi,\om_j(\xi;x,t))\Bigr]^{-1}\bigl(\d_\xi u_i^l\bigr)
(\eta,\om_k(\eta;\xi,\om_j(\xi;x,t))) d \xi d \eta\\
&=&\int_{x_j}^x \int_\eta^x {d}_{jki}(\xi,\eta,x,t)\d_t\om_j(\xi;x,t)\bigl(a_ka_jp_{jk}\bigr)(\xi,\om_j(\xi;x,t))\\
&\times&
\bigl(\d_\xi u_i^l\bigr)(\eta,\om_k(\eta;\xi,\om_j(\xi;x,t))) d \xi d \eta\\
&=&\int_{x_j}^x \int_\eta^x \tilde{d}_{jki}(\xi,\eta,x,t)\bigl(\d_\xi u_i^l\bigr)(\eta,\om_k(\eta;\xi,\om_j(\xi;x,t))) d \xi d \eta\\
\nonumber&=&-\int_{x_j}^x\int_\eta^x \d_\xi \tilde{d}_{jki}(\xi,\eta,x,t)u_i^l\left(\eta,\om_k(\eta;\xi,\om_j(\xi;x,t))\right) d\xi d\eta\\
&+&\int_{x_j}^x\Big[\tilde{d}_{jki}(\xi,\eta,x,t)u_i^l\left(\eta,\om_k(\eta;\xi,\om_j(\xi;x,t))\right)\Big]_{\xi=\eta}^{\xi=x}d\eta.
\end{eqnarray*}
Here
$$
\tilde{d}_{jki}(\xi,\eta,x,t) = {d}_{jki}(\xi,\eta,x,t)\d_t\om_j(\xi;x,t)
\left(
a_ka_jp_{jk}
\right)
(\xi,\om_j(\xi;x,t)).
$$
Now, the desired convergence follows from \reff{eq:lim_0}.

In the second step we prove that there exists $T_2>T_1$ such that $\d_tu$ restricted to $\overline\Pi_{T_2}$
is $C^1$-smooth in $t$ on $\overline\Pi_{T_2}$.
 Once this is done, we differentiate (\ref{eq:1}) with respect to $t$ and get $\d_{xt}^2u\in C\left(\overline\Pi_{T_2}\right)^n$; 
differentiating (\ref{eq:1}) with respect to $x$, we get $\d_{x}^2u\in C\left(\overline\Pi_{T_2}\right)^n$. 
We will be able to conclude that $u\in C^2\left(\overline\Pi_{T_2}\right)^n$,
as desired. To prove the existence of  $T_2$, let $v=\d_tu$. Differentiation of
(\ref{eq:1}) formally in $t$ leads to
\begin{eqnarray*}
(\partial_t  + a_j\partial_x)v_j + \sum_{k=1}^nb_{jk} v_k + \sum_{k=1}^n\d_tb_{jk} u_k + \d_ta_j\partial_xu_j = \d_tf_j.
\end{eqnarray*}
Combining this with (\ref{eq:1}), we obtain
\begin{eqnarray}
\label{e1}
\lefteqn{(\partial_t  + a_j\partial_x)v_j + \sum_{k=1}^nb_{jk} v_k - \frac{\d_ta_j}{a_j}v_j}\nonumber \\
&&= \d_tf_j - \sum_{k=1}^n\d_tb_{jk} u_k 
+ \frac{\d_ta_j}{a_j}\left(\sum_{k=1}^nb_{jk}u_j-f_j\right)=G_j(f_j,\d_tf_j,u).
\end{eqnarray}
Here, for each $j\le n$, $G_j$  is a certain linear function  with smooth coefficients.
Set 
\beq\label{tilde_c}
\tilde c_j(\xi,x,t)=\exp \int_x^\xi
\left(\frac{b_{jj}}{a_{j}}-\frac{\d_ta_j}{a_{j}^2}\right)(\eta,\om_j(\eta;x,t))\,d\eta,\quad
\tilde d_j(\xi,x,t)=\frac{\tilde c_j(\xi,x,t)}{a_j(\xi,\om_j(\xi;x,t))}\nonumber
\ee
and introduce three linear operators $\tilde B, \tilde D, \tilde F: C\left(\overline\Pi_0\right)^n \to C\left(\overline\Pi_0\right)^n$ by
\begin{eqnarray}
\left(\tilde Bu\right)_j(x,t)&=&\tilde c_j(x_j,x,t)u_j\left(x_j,\om_j(x_j;x,t)\right),\label{B1}\\
\left(\tilde Du\right)_j(x,t) & = & -\int_{x_j}^x \tilde d_j(\xi,x,t)\sum_{k=1\atop k\not=j}^n\left(b_{jk}u_k\right)(\xi,\om_j(\xi;x,t))d\xi,\label{D1}\\
\left(\tilde Ff\right)_j(x,t)&=&\int_{x_j}^x \tilde d_j(\xi,x,t)f_j(\xi,\om_j(\xi;x,t))d\xi.\label{F1}
\end{eqnarray}
Similarly to the above, our starting point is that for any $T_2\ge T_1$ the function
$v$ satisfies the following operator equation
resulting  from  (\ref{e1}):
\begin{eqnarray*}
v|_{\overline\Pi_{T_2}}&=&\tilde Bh^\prime+\tilde Dv+\tilde FG(f,\d_tf,u),
\end{eqnarray*}
and, hence, the equation
\beq\label{io_t}
v|_{\overline\Pi_{T_2}}=\tilde Bh^\prime+\tilde D\tilde Bh^\prime+\tilde D^2v+(I+\tilde D)\tilde FG(f,\d_tf,u),
\ee
where $G=(G_1,\dots,G_n)$ and $h^\prime=\left(h_1^\prime,\dots,h_n^\prime\right)$.
Again, due to the assumption \reff{eq:L2}, we can  fix $T_2>T_1$ such that 
the right-hand side of \reff{io_t} does not depend on $u$ and $v$ in  $\overline\Pi\setminus\Pi_{T_1}$. Due to 
Step 1,  the function $(I+\tilde D)\tilde FG(f,\d_tf,u)$ then meets
 the $C^1_t$-regularity. Moreover, $\tilde Bh^\prime+\tilde D\tilde Bh^\prime\in C^\infty$.
 We  are thus left to show that
the operator $\tilde D^2$ is smoothing in the above sense. As $\tilde D$
is exactly the operator $D$ with $c_j$ and $d_j$ replaced by
the smooth functions $\tilde c_j$ and $\tilde d_j$, the desired smoothing property of
$\tilde D^2$ follows from the proof of the smoothness of $D^2$ and the fact that
$\tilde D^2v$ in \reff{io_t} does not depend on $v$ in  $\overline\Pi\setminus\Pi_{T_1}$.

Proceeding further  by induction, assume that, given $r\ge 2$,  
there is $T_{r}>0$ such that $u\in C^{r}\left(\overline\Pi_{T_{r}}\right)^n$
and prove that $u$ meets the $C^{r+1}$-regularity in $t$ on $\overline\Pi_{T_{r+1}}$ for some $T_{r+1}>T_{r}$. 
Set $w=\d_t^ru$. Differentiating (\ref{eq:1}) and (\ref{eq:3}) $r$-times in $t$, we come to our starting
operator equation for $w$, namely
\beq\label{io_rt}
\begin{array}{cc}
w|_{\overline\Pi_{T_{r+1}}}=\tilde Bh^{(r)}+\tilde D\tilde Bh^{(r)}+\tilde D^2w
+(I+\tilde D)\tilde F\tilde G(f,\d_tf,\dots,\d_t^rf,u,\d_tu,\dots,\d_t^{r-1}u), 
\end{array}
\ee
where $\tilde G$ is a vector of certain linear functions   with smooth coefficients
and the operators $\tilde B,\tilde D$, and $\tilde F$ are modified by $\tilde c_j(\xi,x,t)$ in \reff{B1}--\reff{F1}
changing  to $\tilde c_j(\xi,x,t)=\exp \int_x^\xi
\left(\frac{b_{jj}}{a_{j}}-r\frac{\d_ta_j}{a_{j}^2}\right)(\eta,\om_j(\eta;x,t))\,d\eta$. 
Similarly to the above, fix $T_{r+1}>T_{r}$ such that 
the right-hand side of \reff{io_rt} does not depend on $u,\d_tu,\dots,\d_t^{r-1}u$, 
and $w$ in  $\overline\Pi\setminus\Pi_{T_{r}}$.
This ensures that the last two summands in \reff{io_rt} are  $C_t^1$-functions. 
The first two summands are $C_t^1$-smooth by the regularity assumptions on the data.
Finally, the $C^{r+1}\left(\overline\Pi_{T_{r+1}}\right)$-regularity of $u$ follows from 
the previous steps of the proof and
suitable differentiations of the system (\ref{eq:1}).
\end{proof}

Theorem~\ref{thm:classical} can be extended over the boundary operators of the 
following kind (both linear and nonlinear). Given $T>0$, in the domain $\Pi_T$
let us consider the problem   (\ref{eq:1})--(\ref{eq:3})  with 
$b_{jk}\equiv 0$ for all $j\ne k$
(i.e., the system (\ref{eq:1}) is decoupled) and with (\ref{eq:2}) 
replaced by $u(x,T)=\vphi(x)$ (the initial values are given at $t=T$). 
This entails that the domain of influence of  $\vphi$ now depends only on the 
boundary conditions. For the latter it is supposed  that, for every $T>0$ and $\vphi(x)$, 
the function $\vphi(x)$ has a bounded domain of
influence on $u$. In other words, for any decoupled system (\ref{eq:1}), 
if $\vphi(x)$ has a singularity at some point $x\in[0,1]$, then this singularity 
 expands  along a
finite number of characteristic curves (we have a finite number of "reflections"
from the boundary), and this number is bounded from above uniformly  in $x\in[0,1]$.
 This class of boundary operators is 
in detail described in~\cite{Km}, where the necessary and sufficient conditions for smoothing solutions
are given. The results of~\cite{Km} generalize the smoothing results obtained in~\cite{Elt,Hillen,LavLyu,Lyulko} 
for the system \reff{eq:1} with time-independent coefficients and (a kind of) Dirichlet boundary conditions.

\subsection{Integral  boundary conditions in age structured population models}

Here we address another class of boundary operators  admitting smoothing solutions. 
Though it covers a range of (partial) integral operators, we illustrate our smoothing result 
with an example from population dynamics.

Integral boundary conditions are usually used in continuous age structured population models to 
describe a fertility of populations. Let $u(x,t)$ denote the density of a population of age $x$ at time $t$.
Then the dynamics of $u$ can be described by the following model (see, e.g.~\cite{HadeDiet,Magal,webb} and references therein):
\begin{eqnarray}
(\partial_t  + \partial_x + \mu)
 u &=& 0, \quad (x,t)\in\overline\Pi_0,\label{eq:1ex}
\\
u(x,0) &=& \varphi(x),
\quad\ \, x\in [0,1],
\label{eq:2ex}\\
u(0,t)& =& h\left(\int_0^1\ga(x)u(x,t)\, dx\right),\quad t\in\R,\label{eq:3ex}
\end{eqnarray}
where $\mu>0$ is the mortality rate of the population and the functions $h$ and $\ga$ describe the fertility of the population.
Without losing potential applicability to the topic of population dynamics, $h$ and $\ga$ are supposed to be  $C^\infty$-smooth functions. The integral in 
(\ref{eq:3ex}) is a kind of the so-called ``partial'' integral, since $u$ depends not only on the variable of integration $x$,
but also on the free variable~ $t$. Therefore the right-hand side of (\ref{eq:3ex}) is not 
smoothing. Nevertheless,  it turns out that it is regular enough to contribute into smoothing solutions. 

\begin{thm}\label{thm:population}
 Assume that  $h$ and $\ga$ are $C^\infty$-smooth 
functions and $\vphi$ is a continuous function.
Then any continuous solution to the problem (\ref{eq:1ex})--(\ref{eq:3ex}) 
is smoothing.
\end{thm}

\begin{proof}
It suffices to show the smoothing property starting
from large enough $t$. Therefore, we can use the notation:
\begin{eqnarray*}
(Ru)(t)&=&h\left(\int_0^1\ga(x)u(x,t)\, dx\right)\\
\om(\xi;x,t)&=&t+\xi-x\\
c(\xi,x,t)
=\tilde c(\xi,x,t)
&=&e^{\mu(\xi-x)}\\
(Bu)(x,t)=(\tilde Bu)(x,t)&=&e^{-\mu x}u(0,t-x),
\end{eqnarray*}
the latter two being introduced for all large enough $t$. 
Integration along the characteristic curves implies that  any continuous solution to (\ref{eq:1ex})--(\ref{eq:3ex}) satisfies
the operator equations
$u=BRu$ and $u=Bu$ 
and, hence, 
\begin{eqnarray}
u=BRBu\label{pop1}
\end{eqnarray}
whenever  $t>T_1$, where $T_1$ is chosen to be so large that the operator $BRB$ moves away from the initial 
boundary (the right-hand side of \reff{pop1} does not depend on $\vphi$). Since
\begin{eqnarray*}
(BRBu)(t)& =& e^{-\mu x}h\left(\int_0^1\ga(\xi)e^{-\mu\xi}u(0,t-x-\xi)\, d\xi\right)\\
& =& e^{-\mu x}h\left(\int_{t-x-1}^{t-x}\ga(t-x-\tau)e^{\mu(x-t+\tau)}u(0,\tau)\, d\tau\right),
\end{eqnarray*}
we obtain the $C^1_t$-smoothness of $BRBu$ and, hence, of $u$ on $\overline\Pi_{T_1}$.
The $C^1$-smoothness of $u$ on $\overline\Pi_{T_1}$ now follows  from (\ref{eq:1ex}).

Proceeding similarly to the proof of Theorem~\ref{thm:classical}, in the second step we consider
the following operator equation with respect to $v=\d_tu$, obtained after differentiation of (\ref{eq:1ex}) and (\ref{eq:3ex})
with respect to $t$ and integration along characteristic curves:
\begin{eqnarray}
v|_{\overline\Pi_{T_2}}=B\d_tRBv,\label{pop2}
\end{eqnarray}
where 
\begin{eqnarray*}
(\d_tRv)(t) =  h^\prime\left(\int_0^1\ga(x)u(x,t)\, dx\right)\int_0^1\ga(x)v(x,t)\, dx
\end{eqnarray*}
and $T_2>T_1$ is fixed to satisfy the property that the right-hand side of \reff{pop2} does not depend on $u$
and $v$ in $\overline\Pi_0\setminus\Pi_{T_1}$. 
It follows that
\begin{eqnarray*}
v|_{\overline\Pi_{T_2}}& =& e^{-\mu x}h^\prime\left(\int_0^1\ga(\xi)u(\xi,t-x)\, d\xi\right)
\int_0^1\ga(\xi)e^{-\mu\xi}v(0,t-x-\xi)\, d\xi\\
& =& e^{-\mu x}h^\prime\left(\int_0^1\ga(\xi)u(\xi,t-x)\, d\xi\right)
\int_{t-x-1}^{t-x}\ga(t-x-\tau)e^{\mu(x-t+\tau)}v(0,\tau)\, d\tau.
\end{eqnarray*}
To conclude that $v\in C_t^1\left(\overline\Pi_{T_2}\right)^n$, it remains to note that
$u$ under the first integral in the right-hand side  meets the $C_t^1$-regularity,
while the second integral gives us the desired smoothing property.

In general, given $T_{r}$ for $r\ge 2$, we choose $T_{r+1}>T_{r}$
by the argument as above and for $w=\d_t^ru$ have the equation
\begin{eqnarray}
w|_{\overline\Pi_{T_{r+1}}}=B\d_t^rRBw,\label{popr}
\end{eqnarray}
where
\begin{eqnarray*}
(\d_t^rRw)(t) &=&  h^\prime\left(\int_0^1\ga(x)u(x,t)\, dx\right)\int_0^1\ga(x)w(x,t)\, dx\\
&+&\frac{d^{r-1}}{dt^{r-1}}\left[h^\prime\left(\int_0^1\ga(x)u(x,t)\, dx\right)\right]\int_0^1\ga(x)\d_tu(x,t)\, dx\\
&+&\frac{d^{r-2}}{dt^{r-2}}\left[h^\prime\left(\int_0^1\ga(x)u(x,t)\, dx\right)\int_0^1\ga(x)\d_tu(x,t)\, dx\right].
\end{eqnarray*}
Substituting the latter into \reff{popr} and changing variables under the integral of $w$ similarly
to the first two steps, we get the desired 
smoothing property for $w$.
This completes the proof.
\end{proof}

\subsection{Dissipative boundary conditions and periodic problems}\label{sec:dis}

Now we switch to boundary conditions having dissipative nature  and fitting the smoothing property.
A large class of dissipative boundary  conditions
for hyperbolic PDEs is described in~\cite{coron}.

To give an idea of the smoothing effect in this case, consider the following 
specification of  (\ref{eq:1}):
\beq\label{eq:dis}
\begin{array}{ll}
u_j(0,t) = h_j(z(t)), & 1\le j\le m, \ \ 
\\
u_j(1,t) = h_j(z(t)), & m< j\le n, \ \  
\end{array}
\ee
with 
\beq\label{z}
z(t)=\left(u_1(1,t),\dots,u_{m}(1,t),u_{m+1}(0,t),\dots,u_{n}(0,t)\right)\nonumber.
\ee 
In the domain $\Pi_{-\infty}$ we address the problem (\ref{eq:1}), (\ref{eq:dis}) subjected to periodic boundary 
conditions
\beq\label{eq:per}
u(x,t+2\pi)=u(x,t).
\ee
The problems of this kind appear in laser dynamics and chemical kinetics (in Section~\ref{sec:fredh}
we investigate a traveling-wave model of kind (\ref{eq:1}), (\ref{eq:dis}), \reff{eq:per} from laser dynamics).
Within this section, using the standard notation for the (sub)spaces of continuous functions, we assume that
the functions have additional property of  $2\pi$-periodicity in $t$.
Write
$$
h_j^\prime(z)=\nabla_zh_j(z),
\quad h^\prime(z)=\bigl\{\d_kh_j(z)\bigr\}_{j,k=1}^n.
$$

\begin{thm}\label{thm:dissipative}
 Assume that $a_j$, $b_{jk}$, $f_j$, and $h_j$ are 
smooth functions in all their arguments and the conditions   
 (\ref{eq:L1})--\reff{cass} are fulfilled. 
Moreover, the functions $a_j$, $b_{jk}$, $f_j$ are supposed to be $2\pi$-periodic in $t$.
If
\begin{eqnarray}\label{contr2}
\exp \left\{\int_x^{x_j}
\left(\frac{b_{jj}}{a_{j}}-l\frac{\d_ta_j}{a_{j}^2}\right)(\eta,\om_j(\eta;x,t))\,d\eta\right\}
\sum_{k=1}^n\left|\d_kh_j^\prime(z)\right|<1
\end{eqnarray}
for all $j,k\le n$, $x\in[0,1]$, $t\in\R$, $z\in\R^n$, and $l=0,1,\dots,r$,
then any continuous solution to the problem (\ref{eq:1}), (\ref{eq:dis}), (\ref{eq:per}) 
belongs to $C^r(\Pi_{-\infty})$.
\end{thm}

\begin{proof}
Any continuous solution to the problem (\ref{eq:1}), (\ref{eq:dis}), (\ref{eq:per})  in $\Pi_{-\infty}$
fulfills \reff{abstr} with $S=R$ and also satisfies the equation 
\beq\label{abstr_more}
u=Bu+Du+Ff
\ee
where the boundary conditions are not specified. Substituting \reff{abstr_more} into \reff{abstr},
we obtain 
\beq\label{final}
u=BRu+(DB+D^2)u+(I+D)Ff.
\ee
We first show the bijectivity of $I-BR\in\LL\left(C^1_{t}\left(\overline\Pi_{-\infty}\right)^n\right)$.
On the account of \reff{dt} and the definition of $B$  given by \reff{B}, we have
\begin{eqnarray*}
\lefteqn{
(BRu)_j(x,t)=c_j(x_j,x,t)h_j\left(z(\om_j(x_j;x,t))\right)= c_j(x_j,x,t)h_j(0)}\\
&&+\exp\left\{ \int_x^{x_j}
\left(\frac{b_{jj}}{a_{j}}\right)(\eta,\om_j(\eta;x,t))\,d\eta\right\}
\int_0^1h_j^\prime\left(\al z(\om_j(x_j;x,t))\right)\,d\al\cdot z(\om_j(x_j;x,t))
\end{eqnarray*}
and 
\begin{eqnarray*}
\lefteqn{
\d_t\left[(BRu)_j(x,t)\right]=\d_tc_j(x_j,x,t)h_j\left(z(\om_j(x_j;x,t))\right)}\\
&&+h_j^\prime\left(z(\om_j(x_j;x,t))\right)\cdot z^\prime(\om_j(x_j;x,t))
\exp \left\{\int_x^{x_j}
\left(\frac{b_{jj}}{a_{j}}-\frac{\d_ta_j}{a_{j}^2}\right)(\eta,\om_j(\eta;x,t))\,d\eta\right\},
\end{eqnarray*}
where $\cdot$ denotes the scalar product in $\R^n$. Taking into account \reff{z}, 
the bijectivity of $I-BR\in\LL\left(C^1_{t}\left(\overline\Pi_{-\infty}\right)^n\right)$
now follows from the contractibility condition \reff{contr2} with $r=0,1$
and from the proof of the $C^k$-regularity result for solutions of 
first-order hyperbolic PDEs given in~\cite{RauchReed81}.

Now we claim that the operators  $DB$ and $D^2$ in \reff{final} are smoothing.
The latter is smoothing by the proof in Theorem~\ref{thm:classical}. Similar
argument works also for $DB$. 
 Indeed, by the definition of the operators $D$
and $B$ we have 
\begin{eqnarray}
\lefteqn{
\left(DBu^l\right)_j(x,t)}\label{DB}\\
&&=\sum_{k=1\atop k\not=j}^n
\int^{x_j}_x  d_j(\xi,x,t)b_{jk}(\xi,\om_j(\xi;x,t))c_k(x_k,\xi,\om_j(\xi;x,t))u_k^l(x_k,\om_k(x_k;\xi,\om_j(\xi;x,t))) d \xi,
\nonumber
\end{eqnarray}
where the sequence $u^l$ is fixed to satisfy \reff{eq:lim_0} with $\Pi_0$ replaced by $\Pi_{-\infty}$.
To show that $\d_t\left[DBu^l\right]$ converges uniformly on $\overline\Pi_{-\infty}$, we transform the integrals 
in \reff{DB} like to  the case of $D^2$, that is, we differentiate \reff{DB} in $t$, use \reff{cass}, and integrate by parts.
In this way we get the smoothing property for $DB$. Turning back to the formula \reff{final} and
using in addition the fact that  $(I+D)Ff$ is $C^\infty$-smooth, we can 
rewrite \reff{final} in the  equivalent form
\beq\label{ioioio}
u=(I-BR)^{-1}\left[(DB+D^2)u+(I+D)Ff\right],\nonumber
\ee
thereby reaching the $C^1_t$-regularity for $u$.
Afterwards, the $C^1$-regularity of $u$ is a straightforward consequence of the 
system \reff{eq:1}. 

Proceeding similarly to the proof of Theorem~\ref{thm:classical}, we come to the formula for
$v=\d_tu$:
\beq
v=(I-\tilde BR_z^\prime)^{-1}\left[(\tilde D\tilde B+\tilde D^2)v+(I+\tilde D)\tilde FG(f,\d_tf,u)\right],\nonumber
\ee
where $R_z^\prime y=h^\prime(z)y$.  The property that $v\in C^1_t\left(\overline\Pi_{-\infty}\right)^n$
 follows from  the bijectivity of 
$I-BR_z^\prime\in\LL\left(C^1_{t}\left(\overline\Pi_{-\infty}\right)^n\right)$,
which we have by condition \reff{contr2} with $r=1,2$
and the $C_t^1$-regularity of $\tilde DB+\tilde D^2$ and $(I+\tilde D)\tilde FG(f,\d_tf,u)$.
This entails $u\in C^2_t\left(\overline\Pi_{-\infty}\right)^n$. It follows by \reff{eq:1} that
$u\in C^2\left(\overline\Pi_{-\infty}\right)^n$.

To complete the proof, we  proceed by induction on the order of regularity of~$u$. Assume that 
$u\in C^r\left(\overline\Pi_{-\infty}\right)^n$ for some $r\ge 2$ and prove that
$u\in C^{r+1}\left(\overline\Pi_{-\infty}\right)^n$.  
Our starting formula for $w=\d_t^ru$ is as follows:
\begin{eqnarray}
w&=&(I-\tilde BR_z^\prime)^{-1}\Bigl[(\tilde D\tilde B+\tilde D^2)w
\nonumber\\
&+&(I+\tilde D)\tilde F\tilde G(f,\d_tf,\dots,\d_t^rf,u,\d_tu,\dots,\d_t^{r-1}u)\nonumber\\
&+&
\tilde B\d_t^{r-1}R_z^\prime z^\prime+
\tilde B\d_t^{r-2}\left(R_z^\prime z^\prime\right)\Bigr],\nonumber
\end{eqnarray}
where 
$
\d_t^{r-1}R_z^\prime =\bigl\{\d_t^{r-1}(\d_kh_j(z))\bigr\}_{j,k=1}^n
$
 and $\tilde B,\tilde D$, and $\tilde F$ are modified by $\tilde c_j(\xi,x,t)$ in \reff{B1}--\reff{F1}
changing  to $\tilde c_j(\xi,x,t)=\exp \int_x^\xi
\left(\frac{b_{jj}}{a_{j}}-r\frac{\d_ta_j}{a_{j}^2}\right)(\eta,\om_j(\eta;x,t))\,d\eta$.
By the regularity assumptions on the data and the induction assumption, 
the last three summands in the square brackets are $C^1_t$-functions.
Using in addition our smoothing argument for $\tilde D\tilde B+\tilde D^2$ and the regularity properties of $(I-\tilde BR_z^\prime)^{-1}$,
we arrive at the desired conclusion.
\end{proof}

\section{Fredholm solvability of periodic problems}\label{sec:fredh}

In \cite{KmRe1,KR2} we suggest an approach to establish the Fredholm property for
first-order hyperbolic operators. This is done by construction an equivalent
regularization in the form  of a parametrix. The construction is, implicitly but
essentially, based on the smoothing effect investigated in Section~\ref{sec:smooth}. 
Consider the first-order one-dimensional hyperbolic system 
\beq\label{eq:1.1}
(\partial_t  + a(x)\partial_x + b(x))u = f(x,t), \;
  x\in(0,1),
\ee
subjected to periodic conditions \reff{eq:per}
and reflection boundary conditions
\beq\label{eq:1.3}
\begin{array}{l}
\displaystyle
u_j(0,t) = \sum\limits_{k=m+1}^nr_{jk}^0u_k(0,t),\; 1\le j\le m,\\
\displaystyle
u_j(1,t) = \sum\limits_{k=1}^mr_{jk}^1u_k(1,t), \; m<j\le n.\\
\end{array}
\ee
Here $r_{jk}^0$ and $r_{jk}^1$ are real numbers
and the right-hand sides $f_j:[0,1] \times \R \to \R$
are supposed to be $2\pi$-periodic with respect to $t$.

The main result of this section states that the system (\ref{eq:1.1}), \reff{eq:per}, (\ref{eq:1.3}) is solvable if and only if 
the right hand side is orthogonal to all solutions 
to the corresponding homogeneous adjoint system
$$
-\partial_tu  - \partial_x\left(a(x)u\right) + b^T(x)u = 0,
\;  x\in(0,1),  
$$
subjected to periodic conditions \reff{eq:per} and adjoint boundary conditions
\beq \label{eq:1.6}
\begin{array}{l}
\displaystyle
a_j(0)u_j(0,t) = -\sum\limits_{k=1}^mr_{kj}^0a_k(0)u_k(0,t),\; m< j\le n,\\
\displaystyle
a_j(1)u_j(1,t) = -\sum\limits_{k=m+1}^nr_{kj}^1a_k(1)u_k(1,t),\; 1\le j\le m.
\end{array}
\end{equation}

We will  present our result in three steps. First we introduce appropriate function spaces 
for solutions. Then we decompose  the operator of the problem into two parts, only one
being responsible for propagation of singularities. Finally, based on this decomposition 
and the smoothing property, we construct a  parametrix thereby establishing the Fredholm solvability.

When choosing the function spaces, note that the
problem (\ref{eq:1.1}), \reff{eq:per}, (\ref{eq:1.3}) 
describes the so-called traveling-wave models from laser dynamics~\cite{LiRadRe,RadWu}. 
From the physical point of view, it is desirable to allow
discontinuities in the coefficients and the right hand side of \reff{eq:1.1}.
This entails that the spaces of solutions  should not be  too small.
On the other hand, they should not be too large, in order to admit embeddings
into an  algebra of functions with pointwise multiplication of its elements.
The last property is important for potential applicability of our results to 
nonlinear problems, like describing such dynamic phenomena 
as Hopf bifurcation and periodic synchronizations. Finally, the solution spaces 
capable to capture the Fredholm solvability need to have optimal regularity with respect
to the function spaces of the right-hand side.

We now  describe the scale of spaces  $V^\ga$ (for the solutions) and $W^\ga$ (for the right-hand side)
meeting all these properties. 
For $\ga\ge 0$, let $W^{\ga}$ denote the vector space of all locally integrable functions
$f: [0,1]\times\R\to\R^n$ such that
$f(x,t)=f\left(x,t+2\pi\right)$ for almost all $x \in (0,1)$ and $t\in\R$
and that
\begin{equation}\label{eq:1.12}
\|f\|_{W^{\ga}}^2=\sum\limits_{s\in\Z}(1+s^2)^{\gamma}
\int\limits_0^1\left\|\int\limits_0^{2\pi}
f(x,t)e^{-ist}\,dt\right\|^2\,dx<\infty.
\end{equation}
Here and in what follows $\|\cdot\|$ is the Hermitian norm in $\C^n$. 
It is well known that $W^{\ga}$ is a Banach space 
with the norm~(\ref{eq:1.12}); see, e.g.~\cite{herrmann}, \cite[Chapter 5.10]{robinson}, 
and \cite[Chapter 2.4]{vejvoda}.

Furthermore,  for $\ga\ge 1$  and $a \in L^\infty\left((0,1);\M_n\right)$,
where $\M_n$ denotes the space of real $n\times n$ matrices,
 with $\mbox{ess inf } |a_j|>0$ 
for all $j\le n$
we will work with the function spaces
$$
U^{\gamma} = \Bigl\{u\in W^{\gamma}:\, \d_xu\in W^{0},\,
\d_tu+a\d_xu\in W^{\gamma}\Bigr\} 
$$
endowed with the norms
$$
\|u\|_{U^{\gamma}}^2=\|u\|_{W^{\gamma}}^2
+\left\|\d_tu+a\d_xu\right\|_{W^{\gamma}}^2.
$$
Remark that the space $U^{\gamma}$ depends on $a$ and is larger than the space of  all $u \in W^\ga$ such that
$\partial_t u \in  W^\ga$ and $\partial_x u \in  W^\ga$ (which does not  depend on $a$).
For $u \in U^\ga$ there exist traces $u(0,\cdot), u(1,\cdot) \in L^2_{loc}(\R;\R^n)$ (see~\cite{KR2}), and,  hence, it makes sense
to consider the closed subspaces in $U^\ga$
\begin{eqnarray*}
V^\ga&=&\{u \in U^\ga:\, (\ref{eq:1.3}) \mbox{ is fulfilled}\},\\
\tilde{V}^\ga&=&\{u \in U^\ga:\, (\ref{eq:1.6}) \mbox{ is fulfilled}\}.
\end{eqnarray*}

Our next task is to decompose the operator of our problem into two parts in order
to single out the part, denoted below by $\A$, which is bijective and at the same time
is responsible for the propagation of singularities. If this decomposition  is optimal, then 
 after a  regularization procedure the other part becomes smoothing and 
therefore meets the compactness property. 
Let
$$
b^0=\mbox{diag}(b_{11},b_{22},\ldots,b_{nn}) \; \mbox{ and } \; b^1=b-b^0
$$
denote the diagonal and the off-diagonal parts of the coefficient matrix $b$, respectively.
Let us introduce operators $\A\in\LL(V^{\gamma};W^{\gamma})$,
$\tilde{\A}\in\LL(\tilde{V}^{\gamma};W^{\gamma})$,
and $\B,\tilde{\B}\in\LL(W^{\gamma})$ by
\begin{eqnarray*} 
\begin{array}{rcl}
\A u&=&\d_tu+a\d_xu+b^0u,\nonumber\\
\tilde\A u&=&-\d_tu-\d_x(au)+b^0u,\nonumber\\
\B u&=&b^1u,\nonumber\\
\tilde\B u&=&(b^1)^Tu.\nonumber
\end{array}
\end{eqnarray*}
Remark that the operators  $\A$, $\B$, and $\tilde{\B}$ are well-defined for $a_j, b_{jk} \in  L^\infty(0,1)$, while 
$\tilde{\A}$ is well-defined under additional regularity assumptions with respect to the coefficients
$a_j$, for example, 
for $a_j \in C^{0,1}([0,1])$.
Note that the operator equation 
\begin{equation}
\label{abstr1}
\A u+\B u=f\nonumber
\end{equation} 
is an abstract representation of the
periodic-Dirichlet problem~(\ref{eq:1.1}), \reff{eq:per}, (\ref{eq:1.3}).

Finally,  for $s \in \Z$
we introduce
the complex $(n-m)\times(n-m)$ matrices 
\begin{equation}
R_s=
\left[\sum\limits_{l=1}^m
e^{is(\al_j(1)-\al_l(1))+\be_j(1)-\be_l(1)}r_{jl}^1r_{lk}^0\right]_{j,k=m+1}^n,\nonumber
\end{equation}
where
\begin{equation}\label{coef}
\al_j(x)=\int_0^x\frac{1}{a_j(y)}\,dy, \;\;\be_j(x)=\int_0^x\nonumber
\frac{b_{jj}(y)}{a_j(y)}\,dy.
\end{equation}

The following theorem states, first, that the pair of spaces $(V^\ga,W^\ga)$ gives an
optimal regularity trade-off between the spaces of solutions and right-hand sides and,
second, that $\A$ meets the bijectivity property. The second desirable property for 
$\A$ of being an optimal operator responsible for propagation of singularities will be a consequence of
our Fredholmness result.

\begin{thm}\cite{KR2}\label{thm:isom} 
For every $c>0$ there exists $C>0$ such that the following is true:
If 
\begin{equation}\label{ge}
a_j, b_{jj} \in  L^\infty(0,1)\; \mbox{ and } \; \ess\inf|a_j|\ge c \; \mbox{ for all } j=1,\ldots,n,
\end{equation}
\begin{equation}\label{le}
\sum_{j=1}^n\|b_{jj}\|_\infty
+ \sum_{j=1}^m\sum_{k=m+1}^n |r^0_{jk}|+\sum_{j=m+1}^n\sum_{k=1}^m |r^1_{jk}|\le \frac{1}{c},\nonumber
\end{equation}
and
\beq
\label{cond}
|\det (I-R_s)|\ge c \; \mbox{ for all } \; s \in \Z,
\ee
then for all $\ga \ge 1$
the operator $\A$ is an isomorphism from $V^{\gamma}$ onto $W^{\gamma}$ and
$$
\|\A^{-1}\|_{{\mathcal L}(W^{\gamma};V^{\gamma})}\le C.
$$
\end{thm}

Let
\beq \label{eq:1.16}
\langle f,u\rangle_{L^2}=\frac{1}{2\pi}\int_0^{2\pi}\int_0^{1}
\left\langle f(x,t),u(x,t)\right\rangle\,dxdt\nonumber
\ee
denote the scalar product in the Hilbert space $L^2\left((0,1)\times(0,2\pi);\R^n\right)$
and $\langle \cdot,\cdot\rangle$ denote  the Euclidean scalar product in $\R^n$.
As usual, by $BV(0,1)$ we denote the Banach space of all functions $h:(0,1) \to \R$ with bounded variation,
i.e. of
all $h \in L^\infty(0,1)$ such that there exists $C>0$ with
\beq
\label{BV}
\left|\int_0^1 h(x)\vphi'(x) dx \right| \le C \|\vphi\|_{L^\infty(0,
1)} \mbox{ for all }
\vphi \in C_0^\infty(0,1).
\ee
The norm of $h$ in  $BV(0,1)$ is the sum of the norm of  $h$ in  $L^\infty(0,1)$ and of the smallest
possible constant $C$ in (\ref{BV}).
We are prepared to formulate the main result of this section.

\begin{thm}\cite{KR2}\label{thm:fredh} Suppose that conditions \reff{ge}  and \reff{cond}  are fulfilled for some $c>0$.
Suppose also that
\beq \label{eq:1.10}
\begin{array}{l}
\mbox{for all } j\ne k \mbox{ there is } p_{jk}\in BV(0,1) \mbox{ such that }\\
\displaystyle a_k(x)b_{jk}(x)a = p_{jk}(x)(a_j(x)-a_k(x)) \mbox{ for a.a. } x \in [0,1].
\end{array}
\ee
Then the following is true:

(i) The operator $\A+\B$ is a Fredholm operator with index zero
from $V^{\gamma}$ into $W^{\gamma}$ for all $\ga\ge 1$, and 
$
\ker({\A}+{\B})=\left\{u \in V^{\gamma}:\,\left(\A+\B\right)u=0\right\}
$
does not depend on~$\ga$.

(ii) {\rm (smoothing effect)}  If $a \in C^{0,1}\left([0,1];\M_n\right)$, then 
$\ker({\A}+{\B})^*=\ker({\tilde{\A}}+{\tilde{\B}})$ and 
\begin{eqnarray*}
\left\{(\A+\B)u: u \in V^{\gamma}\right\}
=\left\{f\in W^{\ga}:\langle f,u\rangle_{L^2}=0
\mbox{ for all }u\in\ker(\tilde\A+\tilde \B)\right\},
\end{eqnarray*}
where
$
\ker({\tilde{\A}}+{\tilde{\B}})=\{u \in \tilde{V}^{\gamma}:\,
(\tilde{\A}+\tilde{\B})u=0\}
$
does not depend on $\ga$.
\end{thm}

 Theorem~\ref{thm:fredh} (ii) states that the kernel of the adjoint operator is 
actually defined on the classical function spaces. In other words, 
the kernel has much better regularity than ensured just by the formal definition of the adjoint operator.
Here we encounter a smoothing effect for the solutions (of the adjoint hyperbolic problem), 
that are originally functionals. The proof of  this effect in~\cite{KmRe1,KR2} uses completely 
different techniques, based on a functional-analytic approach.

Finally, we  outline the proof of Theorem~\ref{thm:fredh} (i). As mentioned above, we
construct a parametrix to the operator of the problem. By Theorem~\ref{thm:isom},
the zero-order Fredholmness of the operator $\A+\B\in\LL(V^\ga;W^\ga)$ is equivalent to 
the zero-order Fredholmness of the operator $I+\B\A^{-1}\in\LL(W^\ga)$. Furthermore,
we use the following Fredholmness criterion (see also~\cite[Theorem 5.5]{Schechter}
or \cite[Proposition 5.7.1]{Zeidler}).

\begin{lem}\cite{KmRe1}\label{lem:criter}
Let $I$ denote the identity in a Banach space $W$. Suppose that $\D\in\LL(W)$ and $\D^2$
is compact.  Then $I+\D$ is Fredholm.
\end{lem}

Setting $\D=\B\A^{-1}\in\LL(W^\ga)$, we prove that $\D^2\in\LL(W^\ga)$ is compact
(while $\D$ alone can hardly be compact, being a type of a partial integral operator).
This actually means that $\D^2$ has smoothing property. In fact, $\D^2$ is basically the same 
as the operator $D^2$, that we used in the proof of Theorem~\ref{thm:classical}.

Since $I-\D^2=(I-\D)(I+\D)=(I+\D)(I-\D)$,  the operator $I-\D$
is a parametrix of $I+\D$. It follows that the operator $\A+\B$ admits an
equivalent regularization in the form of the right parametrix $\A^{-1}(I-\B\A^{-1})$.

\section*{Acknowledgments}
This work  was 
supported by the Alexander von Humboldt Foundation and  
the DFG Research Center {\sc Matheon}
mathematics for key technologies (project D8).

\end{document}